\documentclass[12pt]{scrartcl}

\usepackage[]{amsmath, amssymb,amsfonts,amsthm,mathtools,braket,color,enumerate}

\usepackage{here}
\usepackage{tikz-cd}
\usepackage{tikz,subfigure}
\usetikzlibrary{decorations}
\usetikzlibrary{arrows}
\tikzstyle{v} = [circle, draw, inner sep=2pt, minimum size=3pt, fill=black]
\tikzstyle{e} = [fill, opacity=.2, 
fill opacity=.2, 
line join=round, 
line width=13pt]
\pgfdeclarelayer{background}
\pgfsetlayers{background,main}

\pgfdeclaredecoration{dashsoliddouble}{initial}{
  \state{initial}[width=\pgfdecoratedinputsegmentlength]{
    \pgfmathsetlengthmacro\lw{.3pt+.5\pgflinewidth}
    \begin{pgfscope}
      \pgfpathmoveto{\pgfpoint{0pt}{\lw}}%
      \pgfpathlineto{\pgfpoint{\pgfdecoratedinputsegmentlength}{\lw}}%
      \pgfmathtruncatemacro\dashnum{%
        round((\pgfdecoratedinputsegmentlength-3pt)/6pt)
      }
      \pgfmathsetmacro\dashscale{%
        \pgfdecoratedinputsegmentlength/(\dashnum*6pt + 3pt)
      }
      \pgfmathsetlengthmacro\dashunit{3pt*\dashscale}
      \pgfsetdash{{\dashunit}{\dashunit}}{0pt}
      \pgfusepath{stroke}
      \pgfsetdash{}{0pt}
      \pgfpathmoveto{\pgfpoint{0pt}{-\lw}}%
      \pgfpathlineto{\pgfpoint{\pgfdecoratedinputsegmentlength}{-\lw}}%
      \pgfusepath{stroke}
    \end{pgfscope}
  }
}

\theoremstyle{plain}
\newtheorem{theorem}{Theorem}[section]
\newtheorem{lemma}[theorem]{Lemma}
\newtheorem{proposition}[theorem]{Proposition}
\newtheorem{corollary}[theorem]{Corollary}

\theoremstyle{definition}
\newtheorem{definition}[theorem]{Definition}

\newtheorem{example}[theorem]{Example}

\newtheorem{remark}[theorem]{Remark}
\newtheorem{question}[theorem]{Question}

\DeclareMathOperator{\Der}{Der}

\DeclareMathOperator{\expon}{exp}

\DeclareMathOperator{\Cox}{Cox}
\DeclareMathOperator{\Shi}{Shi}
\DeclareMathOperator{\Cat}{Cat}
 \newcommand \A{{\mathcal A}}
  \newcommand \B{{\mathcal B}}

\title{Freeness of hyperplane arrangements associated with gain graphs}

\author{
Daisuke Suyama
\thanks{
Department of Electronics and Information Engineering, Hokkai-Gakuen University, Hokkaido 062-8605, Japan. 
E-mail: dsuyama@hgu.jp
} \and
Michele Torielli
\thanks{Department of Mathematics \& Statistics, Northern Arizona University, 801 S Osborne Drive, Flagstaff, AZ 86011, USA. 
E-mail: michele.torielli@nau.edu
} \and
Shuhei Tsujie
\thanks{Department of Mathematics, Hokkaido University of Education, Asahikawa, Hokkaido 070-8621, Japan. 
E-mail: tsujie.shuhei@a.hokkyodai.ac.jp}
}

\date{}

\begin{document}
\maketitle
	
\begin{abstract}
Athanasiadis studied arrangements obtained by adding shifted hyperplanes to the braid arrangement. 
Similarly, Bailey studied arrangements obtained by adding tilted hyperplanes to the braid arrangement. 
These two kinds of arrangements are associated with directed graphs and their freeness was characterized in terms of the associated graphs. 
In addition, there is coincidence of freeness. 
Namely, if Athanasiadis' arrangement is free, then the corresponding Bailey's arrangement is free, and vice versa. 

In this paper, we generalize this phenomenon by using gain graphs. 
\end{abstract}



\section{Introduction}\label{sec:introduction}

A \textbf{hyperplane arrangement} is a finite collection of affine hyperplanes in a finite dimensional vector space. 
If every hyperplane in an arrangement goes through the origin, we call the arrangement \textbf{central}. 

One of the interesting properties in the study of central hyperplane arrangements is their \textbf{freeness} (see Definition \ref{definition freeness}). 
Although freeness is a highly algebraic property, it is closely related to combinatorial properties. 
Terao's conjecture, whether freeness is determined by the combinatorial data of intersections of hyperplanes, remains open. 
There are two important known classes of free arrangements called \textbf{inductively free} and \textbf{divisionally free} arrangements. 
The conditions for inductive and divisional freeness are combinatorial and hence Terao's conjecture hold for these subclasses.

Another remarkable property of hyperplane arrangements concerning freeness is supersolvability. 
A central arrangement is said to be \textbf{supersolvable} if its intersection poset contains a maximal chain consisting of modular elements (see \cite{stanley1972supersolvable-au} for more details). 
Note that there are the following inclusions for central arrangements (see \cite[Theorem 4.2]{jambu1984free-aim} and \cite[Theorem 4.4(2)]{abe2016divisionally-im} ). 
\begin{align*}
\{\text{supersolvable}\} 
\subsetneq \{\text{inductively free}\}
\subsetneq \{\text{divisionally free}\}
\subsetneq \{\text{free}\}
\end{align*}

Let $\Gamma=([\ell], E_{\Gamma})$ be an acyclic digraph (directed graph) on $[\ell] \coloneqq \{1, \dots, \ell\}$ with directed edge set $E_{\Gamma}$.  
Athanasiadis and Bailey studied (inductive) freeness and supersolvability of $\mathcal{A}(\Gamma)$ and $\mathcal{B}(\Gamma)$ in $\mathbb{C}^{\ell}$ respectively, which are defined by 
\begin{align*}
\mathcal{A}(\Gamma) &\coloneqq \Cox(\ell) \cup \Set{\{x_{i}-x_{j} = 1\} | (i,j) \in E_{\Gamma}}, \\
\mathcal{B}(\Gamma) &\coloneqq \Cox(\ell) \cup \Set{\{x_{i} = 0\} | i \in [\ell]} \cup \Set{\{x_{i}-q x_{j} = 0\} | (i,j) \in E_{\Gamma}}, 
\end{align*}
where $q$ is a fixed element in $\mathbb{C}^{\times}$ which is not a root of unity and $\Cox(\ell)$ denotes the \textbf{Coxeter arrangement of type $\mathrm{A}_{\ell-1}$} (or the \textbf{braid arrangement}) defined by 
\begin{align*}
\Cox(\ell) \coloneqq \Set{\{x_{i}-x_{j} = 0\} | 1 \leq i < j \leq \ell}. 
\end{align*}

Note that when $E_{\Gamma} = \Set{(i,j) | 1 \leq i < j \leq \ell }$ the arrangement $\mathcal{A}(\Gamma)$ is called the \textbf{Shi arrangement} and if $E_{\Gamma}$ a subset of $\Set{(i,j) | 1 \leq i < j \leq \ell }$, then $\mathcal{A}(\Gamma)$ interpolates between the Coxeter arrangement $\Cox(\ell)$ and the Shi arrangement. 

Given an arrangement $\mathcal{A}$ in $\ell$-dimensional space, we can obtain its \textbf{cone} $\mathbf{c}\mathcal{A}$ in a $(\ell+1)$-dimensional space by adding the hyperplane at infinity and homogenizing all the defining equations of the hyperplanes in $\mathcal{A}$. 
For example 
\begin{align*}
\mathbf{c}\mathcal{A}(\Gamma) &\coloneqq \{z=0\} \cup \Set{\{x_{i}-x_{j} = 0\} | 1 \leq i < j \leq \ell} \cup \Set{\{x_{i}-x_{j} = z\} | (i,j) \in E_{\Gamma}},
\end{align*}
where $z$ denotes an additional coordinate and $\{z=0\}$ is the hyperplane at infinity. 

Athanasiadis and Bailey characterized freeness and supersolvability of $\mathbf{c}\mathcal{A}(\Gamma)$ and $\mathcal{B}(\Gamma)$ as follows. 

\begin{theorem}[{\cite[Theorem 4.1]{athanasiadis1998free-ejoc}}, {\cite[Corollary 7.4]{bailey1997tilings}}]\label{Athanasiadis-Bailey free}
The following are equivalent. 
\begin{enumerate}[(1)]
\item $\mathbf{c}\mathcal{A}(\Gamma)$ is (inductively) free. 
\item $\mathcal{B}(\Gamma)$ is (inductively) free. 
\item $\Gamma$ does not have any of two digraphs in Figure \ref{fig:Athanasiadis-Bailey obstruction to freeness} as an induced subgraph. 
\end{enumerate}
\end{theorem}

\begin{theorem}[{\cite[Theorem 4.2]{athanasiadis1998free-ejoc}}, {\cite[Theorem 7.8]{bailey1997tilings}}]\label{Athanasiadis-Bailey supersolvable}
The following are equivalent. 
\begin{enumerate}[(1)]
\item $\mathbf{c}\mathcal{A}(\Gamma)$ is supersolvable. 
\item $\mathcal{B}(\Gamma)$ is supersolvable. 
\item All the edges of $\Gamma$ have the same terminal vertex or have the same initial vertex. 
\end{enumerate}
\end{theorem}

\begin{figure}[t]
\centering
\begin{tikzpicture}
\draw (0,0) node[v](1){};
\draw (.5,1) node[v](2){};
\draw (1,0) node[v](3){};
\draw[-stealth] (1)--(2);
\draw[-stealth] (2)--(3);
\end{tikzpicture}
\hspace{30mm}
\begin{tikzpicture}
\draw (0,1) node[v](1){};
\draw (0,0) node[v](2){};
\draw (1,1) node[v](3){};
\draw (1,0) node[v](4){};
\draw[-stealth] (1)--(2);
\draw[-stealth] (3)--(4);
\end{tikzpicture}
\caption{Obstructions to freeness for $\mathbf{c}\mathcal{A}(\Gamma)$ and $\mathcal{B}(\Gamma)$ in Theorem \ref{Athanasiadis-Bailey free}.}
\label{fig:Athanasiadis-Bailey obstruction to freeness}
\end{figure}
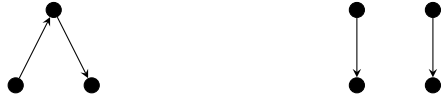

As Athanasiadis \cite[Remark in p.15]{athanasiadis1998free-ejoc} and Bailey \cite[p.105]{bailey1997tilings} pointed out, although the arrangements $\mathbf{c}\mathcal{A}(\Gamma)$ and $\mathcal{B}(\Gamma)$ are different, the coincidences described in Theorem \ref{Athanasiadis-Bailey free} and Theorem \ref{Athanasiadis-Bailey supersolvable} occur. 
In this article, we will discuss these coincidence in more general setting. 

\begin{definition}
A \textbf{simple gain graph} $\Gamma = (V_{\Gamma}, E_{\Gamma}, G_{\Gamma})$ consists of the following data
\begin{itemize}
\item $V_{\Gamma}$ is a finite set. 
\item $G_{\Gamma}$ is a group. 
\item $E_{\Gamma}$ is a finite subset of
\begin{align*}
\Set{(i,j,g) \in V_{\Gamma} \times V_{\Gamma} \times G | i \neq j}/{\sim}, 
\end{align*}
where we use the plus sign $+$ for the operation of $G_{\Gamma}$ and $\sim$ denotes the equivalence relation generated by $(i,j,g) \sim (j,i,-g)$.
We call elements in $V_{\Gamma}, E_{\Gamma}, G_{\Gamma}$, \textbf{vertices}, \textbf{edges} and \textbf{gains}, respectively. 
Let $[i,j,g]$ denote the equivalence class of $(i,j,g)$. 
\end{itemize}
\end{definition}

We simplify the notion of gain graphs for our purpose. 
See \cite{zaslavsky1989biased, zaslavsky1991biased, zaslavsky1995biased, zaslavsky2003biased} for general theory of gain graphs. 

\begin{definition}
Let $\Gamma$ be a simple gain graph on $V_{\Gamma} = [\ell]$ and suppose that the gain group $G_{\Gamma}$ is the additive group of $\mathbb{Z}$ or $\mathbb{F}_{p}$, the finite field of $p$ elements, where $p$ is a prime. 
We define the \textbf{affinographic arrangement} $\mathcal{A}(\Gamma)$ in $\mathbb{Q}^{\ell}$ if $G_{\Gamma} = \mathbb{Z}$ or in $\mathbb{F}_{p}^{\ell}$ if $G_{\Gamma} = \mathbb{F}_{p}$ as follows. 
\begin{align*}
\mathcal{A}(\Gamma) \coloneqq \Set{\{x_{i}-x_{j} = g\} | [i,j,g] \in E_{\Gamma}}. 
\end{align*}
Let $q \in \mathbb{C}^{\times}$ an element which is not a root of unity if $G_{\Gamma} = \mathbb{Z}$ or the primitive $p$-th root of unity if $G_{\Gamma} = \mathbb{F}_{p}$. 
Define the \textbf{bias arrangement} $\mathcal{B}(\Gamma)$ in $\mathbb{C}^{\ell}$ by
\begin{align*}
\mathcal{B}(\Gamma) &\coloneqq \Set{\{x_{i} = 0\} | 1 \leq i \leq \ell} \cup \Set{\{x_{i}-q^{g}x_{j} = 0\} | [i,j,g] \in E_{\Gamma} }. 
\end{align*}
\end{definition}

Note that the definitions of $\mathcal{A}(\Gamma)$ and $\mathcal{B}(\Gamma)$ are well-defined. 
Namely, the hyperplanes $\{x_{i}-x_{j} = g\}$ and $\{x_{i}-q^{g}x_{j} = 0\}$ are independent of the choice of a representative of an edge of $\Gamma$ since 
\begin{align*}
\{x_{i}-x_{j} = g\} = \{x_{j}-x_{i}=-g\} \quad \text{ and } \quad
\{x_{i}-q^{g}x_{j} = 0\} = \{x_{j}-q^{-g}x_{i}=0\}. 
\end{align*}

\begin{example} 
Let $\Gamma$ be the simple gain graph of Figure \ref{Fig:exgaingraph} with $G_{\Gamma} = \mathbb{Z}$ on $3$ vertices consisting of the edges
\begin{align*}
[1,2,0],\ [1,3,0],\ [1,2,1],\ [2,3,1],\ [2,3,-1],\ [1,3,2]. 
\end{align*}
\begin{figure}[t]
\centering
\begin{tikzpicture}
\draw (0,1.73) node[circle, draw, inner sep=1pt](1){$1$};
\draw (-1,0) node[circle, draw, inner sep=1pt](2){$2$};
\draw ( 1,0) node[circle, draw, inner sep=1pt](3){$3$};
\draw[] (1)-- node[xshift=6, yshift=-3] {$0$} (2);
\draw[] (1)-- node[xshift=-6, yshift=-3] {$0$} (3);
\draw[bend right,-stealth] (1) to node[left]{$1$} (2);
\draw[bend left,-stealth] (1) to node[right]{$2$}(3);
\draw[bend left,-stealth] (2) to node[below]{$1$}(3);
\draw[bend right,-stealth] (2) to node[below]{$-1$}(3);
\end{tikzpicture}
\caption{A gain graph with integral gains.}\label{Fig:exgaingraph}
\end{figure}
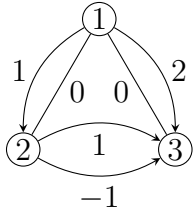
Then the defining polynomial of $\mathcal{A}(\Gamma)$ is 
\begin{align*}
Q_{\mathcal{A}(\Gamma)}=(x_1-x_2)(x_1-x_3)(x_1-x_2-1)(x_2-x_3-1)(x_2-x_3+1)(x_1-x_3-2)
\end{align*}
and the defining polynomial of $\mathcal{B}(\Gamma)$ is 
\begin{align*}
Q_{\mathcal{B}(\Gamma)} = x_{1}x_{2}x_{3}(x_{1}-x_{2})(x_{1}-x_{3})(x_{1}-q x_{2})(x_{2}-q x_{3})(x_{2}-q^{-1} x_{3})(x_{1}-q^{2}x_{3}), 
\end{align*}
where $q \in \mathbb{C}^{\times}$ is not a root of unity. 
\end{example}

\begin{example}
If the gains of all edges in $\Gamma$ are $0$, then $\Gamma$ can be regarded as a simple graph.
The arrangement $\mathcal{A}(\Gamma)$ is known to be a \textbf{graphic arrangement}. 
Moreover, if the edges form a complete graph and $G_{\Gamma} = \mathbb{Z}$, then $\mathcal{A}(\Gamma)$ is the \textbf{Coxeter arrangement of type $\mathrm{A}_{\ell-1}$}. 
Stanley proved that freeness and supersolvability are equivalent in the class of graphic arrangements and they are also equivalent to chordality of $\Gamma$ in any characteristic (see \cite[Corollary 4.10]{stanley2007introduction-gc} and \cite[Theorem 3.3]{edelman1994free-mz}). 
\end{example}

\begin{remark}\label{rem: digraph}
Suppose that $G_{\Gamma} = \mathbb{Z}$ and $E_{\Gamma}$ consists of all edges in $\Set{[i,j,0] | 1 \leq i < j \leq \ell}$ and some edges in $\Set{[i,j,1] | 1 \leq i < j \leq \ell}$. 
Then the edges with gain $1$ forms an acyclic digraph and the arrangements $\mathcal{A}(\Gamma)$ and $\mathcal{B}(\Gamma)$ are the same with the arrangements studied by Athanasiadis and Bailey. 
\end{remark}

\begin{example}\label{ex:Catalan-Shi}
Let $m$ be a positive integer. 
Define the \textbf{extended Catalan arrangements} and the \textbf{extended Shi arrangements} in $\mathbb{Q}^{\ell}$ as follows. 
\begin{align*}
\Cat(\ell, m) &\coloneqq \Set{\{x_{i}-x_{j} = g\} | 1 \leq i < j \leq \ell, -m \leq g \leq m}, \\
\Shi(\ell, m) &\coloneqq \Cat(\ell, m-1) \cup \Set{\{x_{i}-x_{j} = m\} | 1 \leq i < j \leq \ell}, 
\end{align*}
where $\Cat(\ell, 0) \coloneqq \Cox(\ell)$. 
Both these arrangements are affinographic arrangements of gain graphs with gain groups $\mathbb{Z}$ and the cones over them are shown to be inductively free by Edelman and Reiner \cite[Proof of Theorem 3.2]{edelman1996free-dcg} and Athanasiadis \cite[Corollary 3.4]{athanasiadis1998free-ejoc}. 

Arrangements between $\Cat(\ell,m)$ and $\Cat(\ell,m-1)$ are characterized by the edges with gain $m$, that is, the digraph on $[\ell]$ whose arc $(i,j)$ corresponds with the hyperplane $\{x_{i}-x_{j} = m\}$. 
Athanasiadis \cite{athanasiadis2000deformations-asipm} conjectured that these arrangements are free if and only if the corresponding graphs satisfy certain conditions. 
Abe, Nuida, and Numata \cite{abe2009signed-eliminable-jotlms} and Abe \cite{abe2012conjecture-mrl} resolved this conjecture. 

Wang and Jiang \cite{wang2022free-gac} characterized freeness of subarrangements of $\Shi(\ell,1)$ which may not contain all hyperplanes in $\Cox(\ell)$. 

Nakashima and the third author \cite{nakashima2023freeness-dm} extended the class consisting of arrangements between $\Cat(\ell,m)$ and $\Cat(\ell,m-1)$ and characterized freeness in terms of graphs.
As corollaries, it is shown that $\mathrm{c}\Cat(\ell, m)$ is hereditarily free, that is, every restriction of $\mathrm{c}\Cat(\ell, m)$ is free and if $\ell \geq 6$, then $\mathrm{c}\Shi(\ell,m)$ is not hereditarily free. 

In general, it is not easy to determine a concrete basis for the module of logarithmic derivations of free arrangements. 
Yoshinaga and the first author \cite{suyama2021primitive-siagmaa} gave explicit formulas for bases for the extended Catalan and Shi arrangements using with discrete integrals. 
\end{example}

\begin{example}
When the edge set of $\Gamma$ is the empty set, the arrangement $\mathcal{B}(\Gamma)$ is known to be the \textbf{Boolean arrangement}, which can be proven to be supersolvable and hence free easily. 
\end{example}

\begin{example}
Suppose that $G_{\Gamma} = \mathbb{F}_{2}$ and $q = -1$. 
Then the bias arrangement $\mathcal{B}(\Gamma)$ is known as a \textbf{signed graphic arrangement}.  
If $\Gamma$ is complete, that is, $\Gamma$ has all possible edges, then $\mathcal{B}(\Gamma)$ is known as the \textbf{Coxeter arrangement of type $\mathrm{B}_{\ell}$}. 
Freeness and supersolvability for these arrangements are studied in \cite{edelman1994free-mz,zaslavsky2001supersolvable-ejoc,suyama2019signed, torielli2020freeness-tejoc}
\end{example}

\begin{example}
If $G_{\Gamma} = \mathbb{F}_{p}$ and $\Gamma$ is complete, then $\mathcal{B}(\Gamma)$ is the arrangements consisting of reflecting hyperplanes for the complex reflection group $G(p, 1, \ell)$ (the \textbf{full monomial group}) and hence supersolvable. 
Note that every reflection arrangement associated with $G(r,1,\ell)$, where $r$ is a positive integer is supersolvable. 
\end{example}

\begin{example}
Let $\ell=3$, $G_{\Gamma} = \mathbb{Z}$, and $E_{\Gamma} = \Set{[i,j,g] | 1 \leq i < j \leq 3, g \in \{0,1\}}$. 
Then the bias arrangement $\mathcal{B}(\Gamma)$ is given by the following defining polynomial
\begin{align*}
x_{1}x_{2}x_{3}(x_{1}-x_{2})(x_{1}-x_{3})(x_{2}-x_{3})(x_{1}-qx_{2})(x_{1}-qx_{3})(x_{2}-qx_{3}). 
\end{align*}
This is the first example of arrangements which is free and not $K(\pi, 1)$ obtained by Edelman and Reiner \cite[Theorem 2,1]{edelman1995not-botams}. 
They also constructed a family of free bias arrangements including this arrangement \cite[Theorem 3.4]{edelman1996free-dcg}. 
\end{example}

\begin{example}\label{ex:DMS}
Let $G_{\Gamma} = \mathbb{Z}$ and suppose that $E_{\Gamma}$ consists of all edges with gain in $\Set{g \in \mathbb{Z} | -m \leq g \leq m}$. 
Then the corresponding affinographic arrangement $\mathcal{A}(\Gamma)$ is the extended Catalan arrangement as mentioned in Example \ref{ex:Catalan-Shi}. 
Define the \textbf{two-parameter Fuss-Catalan number} or the \textbf{Raney number} $A_{\ell}(s,r)$ by 
\begin{align*}
A_{\ell}(s,r) \coloneqq \dfrac{r}{\ell s+r}\binom{\ell s+r}{\ell}. 
\end{align*}
Note that $A_{\ell}(2,1) = \frac{1}{2\ell+1}\binom{2\ell+1}{\ell} = \frac{1}{\ell+1}\binom{2\ell}{\ell}$ it is the $\ell$-th Catalan number and it is well known that the number of the chambers of the extended Catalan arrangement $\Cat(\ell, m)$ is equal to $\ell ! \cdot A_{\ell}(m+1,1)$. 

Recently, Deshpande, Menon, and Sarkar \cite{deshpande2023refinements-joac} introduced an arrangement whose number of chambers is $\ell! \cdot A_{\ell}(m+1,2)$. 
This arrangement coincides with the bias arrangement $\mathcal{B}(\Gamma)$ and it is determined by the following defining polynomial 
\begin{align*}
x_{1} \cdots x_{\ell} \prod_{\substack{1 \leq i < j \leq \ell \\ -m \leq g \leq m}} (x_{i}-q^{g}x_{j}). 
\end{align*}
\end{example}

From the discussion above, it is clear that both of $\mathcal{A}(\Gamma)$ and $\mathcal{B}(\Gamma)$ form classes including important well-studied arrangements. 
Characterizing freeness of $\mathcal{A}(\Gamma)$ and $\mathcal{B}(\Gamma)$ in terms of a gain graph $\Gamma$ might be an interesting problem. 
The third author \cite{tsujie2020modular-siagmaa} provided a sufficient condition for freeness of these arrangements. 

Zaslavsky characterized supersolvability of $\mathbf{c}\mathcal{A}(\Gamma)$ and $\mathcal{B}(\Gamma)$ in terms of a vertex ordering and, as a result, we obtain the following theorem. 

\begin{theorem}[A consequence of {\cite[Theorems 2.2 and 3.2]{zaslavsky2001supersolvable-ejoc}}]\label{Zaslavsky supersolvability}
If $\B(\Gamma)$ is supersolvable, then $\mathbf{c}\A(\Gamma)$ is supersolvable. If $\Gamma$ is biconnected, then also the converse statement holds true.
\end{theorem}

Note that Theorem \ref{Zaslavsky supersolvability} and Remark \ref{rem: digraph} implies Theorem \ref{Athanasiadis-Bailey supersolvable}. 
To state results concerning freeness we will introduce the notions of inductive and divisional freeness \emph{along edges} for our arrangements $\mathcal{A}(\Gamma)$ and $\mathcal{B}(\Gamma)$ (See Definition \ref{definition IF along edges} and \ref{definition DF along edges}). 
The main results are as follows. 

\begin{theorem}\label{main theorem IF}
Let $\Gamma$ be a simple gain graph. 
Then the following statements are equivalent. 
\begin{enumerate}[(1)]
\item $\mathbf{c}\A(\Gamma)$ is inductively free along edges. 
\item $\B(\Gamma)$ is inductively free along edges.
\end{enumerate}
\end{theorem}

\begin{theorem}\label{main theorem DF}
Let $\Gamma$ be a simple gain graph. 
Then the following statements are equivalent. 
\begin{enumerate}[(1)]
\item $\mathbf{c}\A(\Gamma)$ is divisionally free along edges. 
\item $\B(\Gamma)$ is divisionally free along edges.
\end{enumerate}
\end{theorem}

Unfortunately, Theorem \ref{main theorem IF} and \ref{main theorem DF} cannot imply Theorem \ref{Athanasiadis-Bailey free} directly since there might be an arrangement that is free but not divisionally free in these classes. 
However, if freeness for one class is characterized, then Theorem \ref{main theorem IF} and \ref{main theorem DF} are useful to characterize freeness for the other class. 
We will give an example in the case $G = \mathbb{F}_{2}$ (see Section \ref{sec:signed}). 

We have a natural question below for freeness.  
However, to answer it completely is hard and we may need a breakthrough in the theory of free arrangements. 
\begin{question}\label{Qustion freeness}
Let $\Gamma$ be a simple gain graph. 
Are the following statements equivalent? 
\begin{enumerate}[(1)]
\item $\mathbf{c}\A(\Gamma)$ is free. 
\item $\B(\Gamma)$ is free. 
\end{enumerate}
If these are equivalent, then is there a conceptual reason? 
\end{question}

We will give a few evidence for Question \ref{Qustion freeness}. 

\begin{theorem}\label{main theorem signed}
When $\Gamma$ is a simple gain graph with gain group $G = \mathbb{F}_{2}$, the statement in Question \ref{Qustion freeness} is true. 
\end{theorem}

\begin{theorem}\label{main theorem 3dim}
Let $\Gamma$ be a simple gain graph on $3$ vertices with gain group $G = \mathbb{Z}$. 
Suppose that $q$ is a transcendental number or a positive real number other than $1$. 
Then $\mathrm{c}\mathcal{A}(\Gamma)$ is free if and only if $\mathcal{B}(\Gamma)$ is free. 
\end{theorem}

The organization of this paper is as follows. 

In Section \ref{sec:preliminaries}, we review the theory of free arrangements and gain graphs. 
In Section \ref{sec:proof of main theorems}, we give the proofs of Theorems \ref{main theorem IF} and \ref{main theorem DF}. 
As a corollary, we prove that the arrangement in Example \ref{ex:DMS} is free (Corollary \ref{Fuss-Catalan free}). 
In Section \ref{sec:signed}, we review the characterization of freeness for $\mathcal{B}(\Gamma)$ when $G_{\Gamma} = \mathbb{F}_{2}$ and characterize freeness for the other arrangements $\mathbf{c}\mathcal{A}(\Gamma)$ as an application (Theorem \ref{main theorem signed}). 
In Section \ref{sec:3dim}, we review the theory of freeness of multiarrangements in dimension $2$ and give a proof of Theorem \ref{main theorem 3dim}.

\section{Preliminaries}\label{sec:preliminaries}
\subsection{Characteristic polynomials of hyperplane arrangements}
Let $ \mathbb{K} $ be an arbitrary field, $ \mathcal{A} $ an arrangement in the $ \ell $-dimensional vector space $ \mathbb{K}^{\ell} $. 
Our main reference on the theory of hyperplane arrangement is \cite{orlik1992arrangements}. 

Define the \textbf{intersection poset} $ L(\mathcal{A}) $ of $\mathcal{A}$ by 
\begin{align*}
L(\mathcal{A}) \coloneqq \Set{\bigcap_{H\in \mathcal{B}}H \neq \varnothing | \mathcal{B} \subseteq \mathcal{A}}
\end{align*}
with the partial order defined by reverse inclusion: $X \leq Y \stackrel{\text{def}}{\Leftrightarrow} X \supseteq Y$. 
Note that when $ \mathcal{B} $ is empty, the intersection over $ \mathcal{B} $ is the ambient vector space $ \mathbb{K}^{\ell} $.

The one-variable M\"{o}ebius function $ \mu \colon L(\mathcal{A}) \to \mathbb{Z} $ is defined recursively by 
\begin{align*}
\mu(\mathbb{K}^{\ell}) \coloneqq 1 
\quad \text{ and } \quad 
\mu(X) \coloneqq -\sum_{Y < X}\mu(Y). 
\end{align*}
The \textbf{characteristic polynomial} $\chi(\mathcal{A},t) $  of $\mathcal{A}$ is defined by 
\begin{align*}
\chi(\mathcal{A},t) \coloneqq \sum_{X \in L(\mathcal{A})}\mu(X)t^{\dim X}. 
\end{align*}

\begin{remark}
If $\mathcal{A}$ is central, then $t-1$ divides $\chi(\mathcal{A}, t)$. 
Also note that since $\mathcal{A}(\Gamma)$ is non-essential, that is, every maximal element in $L(\mathcal{A}(\Gamma))$ has dimension greater than $0$, $\chi(\mathcal{A}(\Gamma),t)$ is divisible by $t$. 
\end{remark}

For any $X\in L(\A)$ define the \textbf{localization} of $\A$ to $X$ as the subarrangement $\A_X$ of $\A$ by 
\begin{align*}
\A_X=\{H\in\A~|~X\subseteq H\}.
\end{align*}
For each hyperplane $ H \in \mathcal{A} $, we define the \textbf{restriction} $ \mathcal{A}^{H} $ by 
\begin{align*}
\mathcal{A}^{H} &\coloneqq \Set{H \cap K | K \in \mathcal{A}\setminus\{H\}}. 
\end{align*}
Note that $ \mathcal{A}^{H} $ is an arrangement in $ H $.

\begin{proposition}[{\cite[Corollary 2.57]{orlik1992arrangements}}]
\label{deletion-restriction formula}
Let $\mathcal{A}$ be an arrangement and $H \in \mathcal{A}$. 
Then
\begin{align*}
\chi(\mathcal{A}, t) = \chi(\mathcal{A}^{\prime}, t) - \chi(\mathcal{A}^{H},t), 
\end{align*}
where $\mathcal{A}^{\prime}$ denotes the \textbf{deletion} $\mathcal{A}^{\prime} \coloneqq \mathcal{A}\setminus\{H\}$. 
\end{proposition}

\begin{proposition}[{\cite[Proposition 2.51]{orlik1992arrangements}}]\label{cone}
Let $\mathbf{c}\mathcal{A}$ denote the cone over $\mathcal{A}$. 
Then 
\begin{align*}
\chi(\mathbf{c}\mathcal{A}, t) = (t-1)\chi(\mathcal{A},t). 
\end{align*}
\end{proposition}

\subsection{Freeness of hyperplane arrangements}
In this subsection, we suppose that $\mathcal{A}$ is central. 
Let $ S = \mathbb{K}[x_{1}, \dots, x_{\ell}] $ be the ring of polynomial functions on $ \mathbb{K}^{\ell} $ and $ \Der(S) $ the module of derivations of $ S $. 
Namely, 
\begin{align*}
\Der(S) \coloneqq \Set{\theta \colon S \to S | \text{ $ \theta $ is $ \mathbb{K} $-linear and } \theta(fg) = \theta(f)g + f\theta(g) \text{ for any } f,g \in S}. 
\end{align*}
Given a map $m \colon \mathcal{A} \to \mathbb{Z}_{>0}$,  we call the pair $(\mathcal{A}, m)$ a \textbf{multiarrangement}. 
Define the \textbf{module of logarithmic derivations} $D(\mathcal{A}, m)$ by 
\begin{align*}
D(\A, m) \coloneqq \Set{\theta \in \Der(S) | \theta(\alpha_{H}) \in \alpha_{H}^{m(H)}S \text{ for any } H \in \A}, 
\end{align*}
where $ \alpha_{H} \in (\mathbb{K}^{\ell})^{\ast} $ denotes a defining linear form of a hyperplane $ H \in \A $. 
Note that $ D(\A, m) $ is a graded $ S $-module. 
We identify $\mathcal{A}$ with the multiarrangement $(\mathcal{A},\boldsymbol{1})$, where $\boldsymbol{1}$ is a map identically $1$ and $D(\mathcal{A}) \coloneqq D(\mathcal{A}, \boldsymbol{1})$. 
\begin{definition}\label{definition freeness}
A multiarrangement $(\A, m)$ is said to be \textbf{free with exponents $\expon(\A, m)=(d_1,\dots,d_\ell)$} 
if $D(\A,m)$ is a free $S$-module and $D(\A,m)\cong\bigoplus_{i=1}^\ell S(-d_i)$. 
\end{definition}

\begin{proposition}[{\cite[Corollary 7]{ziegler1989multiarrangements-sci}}{\cite[Proposition1.21]{yoshinaga2014freeness-adlfdsdtm}}]
If $\ell = 2$, then $(\mathcal{A}, m)$ is free and the exponents $(d_{1}, d_{2})$ satisfies $d_{1}+d_{2} = |m|$. 
\end{proposition}

\begin{proposition}[{\cite[Theorem 4.37]{orlik1992arrangements}}, {\cite[Proposition 1.7]{abe2009signed-eliminable-jotlms}}]\label{free localization}
Suppose that $(\mathcal{A},m)$ is free. 
Then the localization $(\mathcal{A}_{X}, m|_{\mathcal{A}_{X}})$ is free for any $X \in L(\mathcal{A})$. 
\end{proposition}

\begin{proposition}[{\cite[Corollary 4.47]{orlik1992arrangements}}]\label{restriction theorem}
Let $H \in \mathcal{A}$. 
If $\mathcal{A}$ and $\mathcal{A}\setminus\{H\}$ is free, then the restriction $\mathcal{A}^{H}$ is free. 
\end{proposition}

We can define the characteristic polynomial $\chi((\mathcal{A}, m), t)$ for multiarrangements $(\mathcal{A}, m)$ which generalizes the characteristic polynomial of a simple arrangement, using the Hilbert series of the derivation modules. 

The following theorem shows close relation between freeness and combinatorics of arrangements. 
\begin{theorem}[Factorization Theorem, {\cite{terao1981generalized-im}}, {\cite[Theorem 4.1]{abe2007characteristic-aim}}]\label{factorization theorem}
Suppose that a multiarrangement $(\mathcal{A}, m)$ is free with exponents $(d_{1}, \dots, d_{\ell})$. 
Then 
\begin{align*}
\chi((\mathcal{A},m), t) = (t-d_{1}) \cdots (t-d_{\ell}). 
\end{align*}
\end{theorem}

Now, we give the definitions of inductive and divisional freeness for simple arrangements. 
They are actually free by the Addition-Deletion Theorem \cite[Theorem 4.51]{orlik1992arrangements} and the Division Theorem \cite[Theorem 1.1]{abe2016divisionally-im}. 
\begin{definition}
The class of \textbf{inductively free arrangements} is defined to be the smallest class of arrangements such that the following conditions hold. 
\begin{enumerate}[(1)]
\item The empty arrangements are inductively free. 
\item If there exists $ H \in \A $ such that both $\A\setminus\{H\}$ and $\A^{H}$ are inductively free and $\expon(\A^{H})\subseteq\expon(\A\setminus\{H\})$, then $\A$ is inductively free.
\end{enumerate}
\end{definition}

\begin{definition}
The class of \textbf{divisionally free arrangements} is defined to be the smallest class of arrangements such that the following conditions hold
\begin{enumerate}[(1)]
\item The empty arrangements are divisionally free. 
\item If there exists $ H \in \A $ such that $ \A^{H} $ is divisionally free and $ \chi(\A^{H},t) $ divides $ \chi(\A,t) $, then $ \A $ is divisionally free. 
\end{enumerate}
\end{definition}

\subsection{Deletion and contraction of gain graphs}

Let $\Gamma = (V_{\Gamma}, E_{\Gamma}, G_{\Gamma})$ be a simple gain graph and $e = [i,j, g]$ an edge of $\Gamma$. 
We define the deletion and contraction of $\Gamma$ with respect to $e$. 

Define the \textbf{deletion} $\Gamma \setminus e$ by deleting the edge $e$ from $\Gamma$. 
Namely, $\Gamma \setminus e$ is defined by the following data
\begin{itemize}
\item $V_{\Gamma \setminus e} \coloneqq V_{\Gamma}$. 
\item $E_{\Gamma \setminus e} \coloneqq E_{\Gamma} \setminus \{e\}$. 
\item $G_{\Gamma \setminus e} \coloneqq G_{\Gamma}$. 
\end{itemize}

The edge $e$ has two direction $(i,j,g)$ and $(j,i,-g)$. 
In order to define the contraction, we need to fix a direction of $e$. 
Define the \textbf{contraction} $\Gamma/(i,j,g)$ by identifying the vertex $i$ with the vertex $j$ and by adding the corresponding gains (see Figure \ref{Fig:contraction}). 
Namely, the contraction $\Gamma/(i,j,g)$ consists of the following data
\begin{itemize}
\item $V_{\Gamma/(i,j,g)} \coloneqq V_{\Gamma}\setminus\{i\}$. 
\item $E_{\Gamma/(i,j,g)} \coloneqq \Set{[u,v,h] | u,v \in V_{\Gamma}\setminus \{i\}, \ [u,v,h] \in E_{\Gamma}} \\ 
\hspace*{60mm} \cup \Set{[k,j,h+g] | k \in V_{\Gamma}\setminus\{i\}, \ [k,i,h] \in E_{\Gamma}}$. 
\item $G_{\Gamma/(i,j,g)} \coloneqq G_{\Gamma}$. 
\end{itemize}
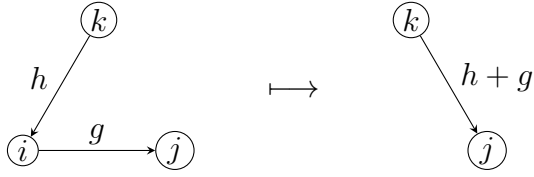
\begin{figure}[t]
\centering
\begin{tikzpicture}[baseline=20]
\draw (0,1.73) node[circle, draw, inner sep=1pt](k){$k$};
\draw (-1,0) node[circle, draw, inner sep=1pt](i){$i$};
\draw ( 1,0) node[circle, draw, inner sep=1pt](j){$j$};
\draw[-stealth] (k)--node[xshift=-8, yshift=3]{$h$}(i);
\draw[-stealth] (i)--node[xshift=0, yshift=6]{$g$}(j);
\end{tikzpicture}
\qquad $\longmapsto$ \qquad
\begin{tikzpicture}[baseline=20]
\draw (0,1.73) node[circle, draw, inner sep=1pt](k){$k$};
\draw ( 1,0) node[circle, draw, inner sep=1pt](j){$j$};
\draw[-stealth] (k)--node[xshift=18, yshift=3]{$h+g$}(j);
\end{tikzpicture}
\caption{Contraction of $(i,j,g)$.}\label{Fig:contraction}
\end{figure}

The deletion and contraction of gain graphs are compatible with the deletion and restriction of the affinographic arrangement $\mathcal{A}(\Gamma)$ and the bias arrangement $\mathcal{B}(\Gamma)$. 
Namely, the following two lemmas hold. 

\begin{lemma}\label{deletion of gain graphs}
Let $e = [i,j,g]$ be an edge of $\Gamma$. Then
\begin{enumerate}[(1)]
\item $\A(\Gamma\setminus e) = \A(\Gamma)\setminus\{x_i-x_j=g\}$, 
\item $\B(\Gamma\setminus e) = \B(\Gamma)\setminus\{x_i-q^{g}x_{j}=0\}$. 
\end{enumerate}
\end{lemma}

\begin{lemma}\label{contraction of gain graphs}
Let $e=[i,j,g]$ be an edge of $\Gamma$ and fix a direction $(i,j,g)$. Then
\begin{enumerate}[(1)]
\item $\A(\Gamma/ (i,j,g)) $ is affinely equivalent to $ \A(\Gamma)^{\{x_i-x_j=g\}}$, 
\item $\B(\Gamma/ (i,j,g)) $ is affinely equivalent to $ \B(\Gamma)^{\{x_i-q^{g}x_{j}=0\}}$, 
\end{enumerate}
where two arrangements $\mathcal{A}_{1}$ and $\mathcal{A}_{2}$ are affinely equivalent if there exists an affine isomorphism $\phi$ between the ambient spaces such that $\mathcal{A}_{2} = \Set{\phi(H) | H \in \mathcal{A}_{1}}$. 
\end{lemma}

\begin{remark}
Since $\{x_{i}-x_{j} = g\} = \{x_{j}-x_{i} = -g\}$, $\mathcal{A}(\Gamma/(i,j,g))$ is affinely equivalent to $\mathcal{A}(\Gamma/(j,i,-g))$ although $\Gamma/(i,j,g)$ is distinct from $\Gamma/(j,i,-g)$ in general. 
By the same reason $\mathcal{B}(\Gamma/(i,j,g))$ is affinely equivalent to  $\mathcal{B}(\Gamma/(j,i,-g))$. 
Thus we will use the notation $\mathcal{A}(\Gamma/e)$ and $\mathcal{B}(\Gamma/e)$ to denote the arrangements corresponding to the contraction. 
\end{remark}

Thanks to Lemma \ref{deletion of gain graphs} and \ref{contraction of gain graphs}, we define the inductive and divisional freeness along edges for $\mathbf{c}\mathcal{A}(\Gamma)$ and $\mathcal{B}(\Gamma)$ as follows.  (We mention the definitions only for $\mathbf{c}\mathcal{A}(\Gamma)$. The definitions for $\mathcal{B}(\Gamma)$ are similar.)

\begin{definition}\label{definition IF along edges}
We say that $\mathbf{c}\mathcal{A}(\Gamma)$ is \textbf{inductively free along edges} if it satisfies the following recursive conditions. 
\begin{enumerate}[(1)]
\item If $E_{\Gamma} = \varnothing$, then $\mathbf{c}\mathcal{A}(\Gamma)$  is inductively free along edges. 
\item If there exists an edge $e \in E_{\Gamma}$ such that $\mathbf{c}\mathcal{A}(\Gamma \setminus e)$ and $\mathbf{c}\mathcal{A}(\Gamma/e)$ are inductively free along edges and $\exp(\mathbf{c}\mathcal{A}(\Gamma/e)) \subseteq \exp(\mathbf{c}\mathcal{A}(\Gamma\setminus e))$, then $\mathbf{c}\mathcal{A}(\Gamma)$ is inductively free along edges. 
\end{enumerate}
\end{definition}

\begin{definition}\label{definition DF along edges}
We say that $\mathbf{c}\mathcal{A}(\Gamma)$ is \textbf{divisionally free along edges} if it satisfies the following recursive conditions. 
\begin{enumerate}[(1)]
\item If $E_{\Gamma} = \varnothing$, then $\mathbf{c}\mathcal{A}(\Gamma)$  is divisionally free along edges. 
\item If there exists an edge $e \in E_{\Gamma}$ such that $\mathbf{c}\mathcal{A}(\Gamma/e)$ are divisionally free along edges and $\chi(\mathbf{c}\mathcal{A}(\Gamma/e),t)$ divides $\chi(\mathbf{c}\mathcal{A}(\Gamma),t)$, then $\mathbf{c}\mathcal{A}(\Gamma)$ is divisionally free along edges. 
\end{enumerate}
\end{definition}

\begin{remark}
If $E_{\Gamma} = \varnothing$, then $\mathbf{c}\mathcal{A}(\Gamma)$ consists of just one hyperplane (the hyperplane at infinity) and $\mathcal{B}(\Gamma)$ is the Boolean arrangement. 
Since they are inductively and divisionally free, inductive and divisional freeness along edges imply inductive and divisional freeness. 
\end{remark}

To study freeness of an arrangement, the characteristic polynomial plays an important role. 
The lemma below states a relation between the characteristic polynomials of $\mathbf{c}\mathcal{A}(\Gamma)$ and $\mathcal{B}(\Gamma)$, which is also proven in \cite[Theorem 2.10 and Remark 2.12]{deshpande2023refinements-joac} when $G_{\Gamma} = \mathbb{Z}$ by using the finite field method. 

\begin{lemma}\label{lem:equalbetweencharpoly}
Let $\Gamma$ be a simple gain graph. 
Then
\begin{align*}
\chi(\A(\Gamma),t)=\chi(\B(\Gamma),t+1)
\end{align*}
and 
\begin{align*}
\chi(\mathbf{c}\A(\Gamma),t)=(t-1)\chi(\B(\Gamma),t+1). 
\end{align*}
\end{lemma}
\begin{proof} 
We will prove this result by induction on $|E_\Gamma|$.
If $|E_\Gamma|=0$, then $\A(\Gamma)$ is the empty arrangement and $\B(\Gamma)$ is the boolean arrangement. This implies that 
$\chi(\A(\Gamma),t)=t^\ell$ and $\chi(\B(\Gamma),t)=(t-1)^\ell$, and hence $\chi(\A(\Gamma),t)=\chi(\B(\Gamma),t+1)$.

Assume $|E_\Gamma|\ge 1$ and consider $e\in E_\Gamma$. Then 
\begin{align*}
\chi(\A(\Gamma),t)&=\chi(\A(\Gamma\setminus e),t)-\chi(\A(\Gamma/ e),t), \\
\chi(\B(\Gamma),t+1)&=\chi(\B(\Gamma\setminus e),t+1)-\chi(\B(\Gamma/ e),t+1)
\end{align*}
by Proposition \ref{deletion-restriction formula}, Lemma \ref{deletion of gain graphs}, and Lemma \ref{contraction of gain graphs}. 
Since $|E_{\Gamma\setminus e}|, |E_{\Gamma/e}|<|E_\Gamma|$, by the induction hypothesis, $\chi(\A(\Gamma),t)=\chi(\B(\Gamma),t+1)$.

Using Proposition \ref{cone}, we have 
\begin{align*}
\chi(\mathbf{c}\A(\Gamma),t) = (t-1)\chi(\mathcal{A},t) = (t-1)\chi(\B(\Gamma),t+1)
\end{align*}
\end{proof}

\begin{corollary}\label{exponents of A and B}
Suppose that $\mathbf{c}\mathcal{A}(\Gamma)$ and $\mathcal{B}(\Gamma)$ are  free and let $\exp(\mathbf{c}\mathcal{A}(\Gamma)) = (0,1,d_{2}, \dots, d_{\ell})$. 
Then $\exp(\mathcal{B}(\Gamma)) = (1,d_{2}+1, \dots, d_{\ell}+1)$. 
\end{corollary}
\begin{proof}
Use Lemma \ref{lem:equalbetweencharpoly} and Theorem \ref{factorization theorem}. 
\end{proof}

\section{Proofs of Theorem \ref{main theorem IF} and \ref{main theorem DF}}\label{sec:proof of main theorems}

\begin{theorem}[Restatement of Theorem \ref{main theorem IF}]
Let $\Gamma$ be a simple gain graph. 
Then the following statements are equivalent
\begin{enumerate}[(1)]
\item $\mathbf{c}\A(\Gamma)$ is inductively free along edges. 
\item $\B(\Gamma)$ is inductively free along edges.
\end{enumerate}
\end{theorem}
\begin{proof}
We will proceed by induction on $|E_{\Gamma}|$. 
If $|E_\Gamma|=0$, then both of $\mathbf{c}\A(\Gamma)$ and $\B(\Gamma)$ are inductively free along edges by definition. 

Suppose that $|E_{\Gamma}| \geq 1$ and $\mathbf{c}\A(\Gamma)$ is inductively free along edges. 
Then there exists $e \in E_{\Gamma}$ such that $\mathbf{c}\mathcal{A}(\Gamma/e)$ and $\mathbf{c}\mathcal{A}(\Gamma\setminus e)$ are inductively free along edges with $\exp(\mathbf{c}\mathcal{A}(\Gamma/e)) \subseteq \exp(\mathbf{c}\mathcal{A}(\Gamma\setminus e))$. 
By the induction hypothesis, $\mathcal{B}(\Gamma/e)$ and $\mathcal{B}(\Gamma \setminus e)$ are inductively free along edges. 
Since $\exp(\mathcal{B}(\Gamma/e)) \subseteq \exp(\mathcal{B}(\Gamma \setminus e))$ by Corollary \ref{exponents of A and B}, we conclude that $\mathcal{B}(\Gamma)$ is inductively free along edges. 
The opposite implication is similar. 
\end{proof}

\begin{theorem}[Restatement of Theorem \ref{main theorem DF}]
Let $\Gamma$ be a simple gain graph. 
Then the following statements are equivalent. 
\begin{enumerate}[(1)]
\item $\mathbf{c}\A(\Gamma)$ is divisionally free along edges. 
\item $\B(\Gamma)$ is divisionally free along edges.
\end{enumerate}
\end{theorem}
\begin{proof}
We will prove this result by induction on $|E_\Gamma|$.
If $|E_\Gamma|=0$, then $\mathbf{c}\A(\Gamma)$ and $\B(\Gamma)$ are both divisionally free along edges by definition.

Consider the case $|E_\Gamma|\ge 1$. Assume that $\mathbf{c}\A(\Gamma)$ is divisionally free along edges, and hence consider $e\in E_\Gamma$ such that $\mathbf{c}\A(\Gamma/ e)$ is divisionally free along edges and $\chi(\mathbf{c}\A(\Gamma/ e),t)$ divides $\chi(\mathbf{c}\A(\Gamma),t)$. By the induction hypothesis and Lemma \ref{lem:equalbetweencharpoly}, we have that $\B(\Gamma/ e)$ is divisionally free along edges and $\chi(\B(\Gamma/ e),t)$ divides $\chi(\B(\Gamma),t)$. This implies that $\B(\Gamma)$ is divisionally free along edges. A similar argument proves the opposite implication.
\end{proof}

As mentioned in Example \ref{ex:Catalan-Shi}, the cones over the extended Catalan and Shi arrangements $\mathrm{c}\Cat(\ell,m)$ and $\mathrm{c}\Shi(\ell,m)$ are inductively free. 
Moreover, according to the proofs, they are inductively free along edges. 
Hence we obtain the following corollaries. 

\begin{corollary}\label{Fuss-Catalan free}
The arrangement determined by
\begin{align*}
x_{1} \cdots x_{\ell} \prod_{\substack{1 \leq i < j \leq \ell \\ -m \leq g \leq m}} (x_{i}-q^{g}x_{j}). 
\end{align*}
is inductively free along edges with exponents $(1, m\ell+2, m\ell+3, m\ell+\ell)$. 
\end{corollary}

\begin{corollary}
The arrangement determined by
\begin{align*}
x_{1} \cdots x_{\ell} \prod_{\substack{1 \leq i < j \leq \ell \\ 1-m \leq g \leq m}} (x_{i}-q^{g}x_{j}). 
\end{align*}
is inductively free along edges with exponents $(1, m\ell+1, m\ell+1, m\ell+1)$. 
\end{corollary}

\section{Proof of Theorem \ref{main theorem signed}}\label{sec:signed}
In this section, we suppose that $G_{\Gamma} = \mathbb{F}_{2}$. 
Since the additive group of $\mathbb{F}_{2}$ is isomorphic to the multiplicative group $\{\pm 1\}$, the gain graph $\Gamma$ is called a \textbf{signed graph}. 
We call an edge with gain $0$ (resp. $1$) a \textbf{positive edge} (resp. \textbf{negative edge}). 
When $\Gamma$ is a signed graph, we call $\mathcal{B}(\Gamma)$ the \textbf{signed graphic arrangement}. 

A cycle in a signed graph $\Gamma$ is called \textbf{balanced} is the number of negative edges in it is even. 
Otherwise, we call it \textbf{unbalanced}. 
A signed graph $\Gamma$ is called \textbf{balanced chordal} if every balanced cycle in $\Gamma$ of length at least four has a chord separating the cycle into two balanced cycles. 

A \textbf{switching} at a vertex $i$ in $\Gamma$ is an operation interchanging the signs of the edges incident to $i$. 
We say that two signed graphs are \textbf{switching equivalent} if one is obtained by applying a finite number of switchings to the other. 
Note that if two signed graphs $\Gamma$ and $\Gamma^{\prime}$ are switching equivalent, then the corresponding signed graphic arrangements $\mathcal{B}(\Gamma)$ and $\mathcal{B}(\Gamma^{\prime})$ are affinely equivalent since a switching at $i$ corresponds to the coordinate change $x_{i} \mapsto -x_{i}$. 
The affinographic arrangements $\mathcal{A}(\Gamma)$ and $\mathcal{A}(\Gamma^{\prime})$ are affinely equivalent since switching at $i$ corresponds to the coordinate change $x_{i} \mapsto x_{i}+1$ and the base field is $\mathbb{F}_{2}$. 

The authors previous works characterized freeness of signed graphic arrangement $\mathcal{B}(\Gamma)$ as follows. 

\begin{theorem}[\cite{suyama2019signed,torielli2020freeness-tejoc}]\label{freeness of signed graphic arrangement}
Let $\Gamma$ be a signed graph. 
Then the following conditions are equivalent. 
\begin{enumerate}[(1)]
\item\label{freeness of signed graphic arrangement 1} $\Gamma$ satisfies the following three conditions. 
\begin{enumerate}[(I)]
\item $\Gamma$ is balanced chordal. 
\item $\Gamma$ has no induced subgraphs isomorphic to unbalanced cycles of length at least three. 
\item $\Gamma$ has no induced subgraphs which are switching equivalent to the graph in Figure \ref{Fig:obstruction signed graph}. 
\begin{figure}[t]
\centering
\begin{tikzpicture}
\draw (0,1) node[v](1){}; 
\draw (0,0) node[v](2){}; 
\draw (1,0) node[v](3){};
\draw (1,1) node[v](4){};
\draw (2)--(1)--(4)--(2)--(3)--(4);
\draw[dashed] (1)--(3);
\draw[dashed] (1) to [bend right] (2);
\draw[dashed] (3) to [bend right] (4);
\end{tikzpicture}
\caption{An obstruction to freeness for the signed graphic arrangement $\mathcal{B}(\Gamma)$. 
(Solid and dashed line segments denote positive and negative edges respectively.)}\label{Fig:obstruction signed graph}
\end{figure}
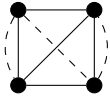
\end{enumerate}
\item \label{freeness of signed graphic arrangement 2} $\mathcal{B}(\Gamma)$ is divisionally free along edges. 
\item $\mathcal{B}(\Gamma)$ is free. 
\end{enumerate}
\end{theorem}

\begin{remark}
In the paper \cite{torielli2020freeness-tejoc}, the condition (\ref{freeness of signed graphic arrangement 2}) in Theorem \ref{freeness of signed graphic arrangement} does not assert ``along edges". 
However, it was shown that $\mathcal{B}(\Gamma)$ is divisionally free along edges if the condition (\ref{freeness of signed graphic arrangement 1}) holds. 
\end{remark}

\begin{theorem}[Restatement of Theorem \ref{main theorem signed}]
Let $\Gamma$ be a signed graph, that is, a simple gain graph with $G_{\Gamma} = \mathbb{F}_{2} \simeq \{\pm 1\}$. 
Then the following conditions are equivalent
\begin{enumerate}[(1)]
\item $\mathbf{c}\mathcal{A}(\Gamma)$ is free. 
\item $\mathcal{B}(\Gamma)$ is free. 
\end{enumerate}
\end{theorem}
\begin{proof}
First, suppose that $\mathcal{B}(\Gamma)$ is free. 
Then $\mathcal{B}(\Gamma)$ is divisionally free along edges by Theorem \ref{freeness of signed graphic arrangement}. 
Therefore $\mathbf{c}\mathcal{A}(\Gamma)$ is divisionally free along edges by Theorem \ref{main theorem DF} and hence free. 

Next, suppose that $\mathbf{c}\mathcal{A}(\Gamma)$ is free. 
In order to show $\mathcal{B}(\Gamma)$ is free, it suffices to prove that $\Gamma$ satisfies the condition (\ref{freeness of signed graphic arrangement 1}) in Theorem \ref{freeness of signed graphic arrangement}. 
Assume that $\Gamma$ does not satisfy the the condition (\ref{freeness of signed graphic arrangement 1}). 
Then $\Gamma$ has an induced subgraph $\Sigma$ satisfying one of the following three conditions. 
\begin{enumerate}[(i)]
\item\label{signed subgraph condition 1} $\Sigma$ is a balanced cycle of length at least four. 
\item\label{signed subgraph condition 2} $\Sigma$ is an unbalanced cycle of length at least three. 
\item\label{signed subgraph condition 3} $\Sigma$ is switching equivalent to the graph in Figure \ref{Fig:obstruction signed graph}. 
\end{enumerate}

Since $\mathbf{c}\mathcal{A}(\Sigma)$ is affinely equivalent to a localization of $\mathbf{c}\mathcal{A}(\Gamma)$, it is sufficient to show that $\mathbf{c}\mathcal{A}(\Sigma)$ is non-free by Proposition \ref{free localization}. 

Consider the case (\ref{signed subgraph condition 1}). 
By switching we may assume that the balanced cycle $\Sigma$ consists of positive edges. 
Then $\mathcal{A}(\Sigma)$ is a graphic arrangement (over $\mathbb{F}_{2}$) of the non-chordal graph $\Sigma$ and hence $\mathbf{c}\mathcal{A}(\Sigma)$ is non-free. 

Next, consider the case (\ref{signed subgraph condition 2}). 
Suppose that the length of the unbalanced cycle $\Sigma$ is three. 
Then the characteristic polynomial of $\mathcal{A}(\Sigma)$ is 
\begin{align*}
\chi(\mathcal{A}(\Sigma),t) = t(t^{2}-3t+3). 
\end{align*}
By Theorem \ref{factorization theorem}, $\mathbf{c}\mathcal{A}(\Sigma)$ is non-free. 
Assume that the length of $\Sigma$ is at least four and $\mathbf{c}\mathcal{A}(\Sigma)$ is free. 
Chose an edge $e$ of $\Sigma$. 
Then the deletion $\Sigma \setminus e$ is switching equivalent to a path consisting of positive edges and hence $\mathbf{c}\mathcal{A}(\Sigma\setminus e)$ is free. 
However, the contraction $\Sigma/e$ is an unbalanced cycle of length at least three and $\mathbf{c}\mathcal{A}(\Sigma/ e)$ is non-free by induction. 
This contradicts to Proposition \ref{restriction theorem} and hence $\mathbf{c}\mathcal{A}(\Sigma)$ is non-free. 

Finally, consider the condition (\ref{signed subgraph condition 3}). 
Then 
\begin{align*}
\chi(\mathcal{A}(\Sigma),t) = t(t-2)(t^{2}-6t+10). 
\end{align*}
By Theorem \ref{factorization theorem}, $\mathbf{c}\mathcal{A}(\Sigma)$ is non-free. 
\end{proof}

Edelman and Reiner \cite[Theorem 4.6]{edelman1994free-mz} characterized freeness of subarrangements between Weyl arrangements of type $\mathrm{A}_{\ell-1}$ and $\mathrm{B}_{\ell}$ in terms of signed graphs. 
From this result, we have the following corollary. 
\begin{corollary}\label{Edelman-Reiner}
Let $\Gamma$ be a signed graph and $\Gamma_{+}$ and $\Gamma_{-}$ denote the simple graphs consisting of positive and negative edges. 
Suppose that $\Gamma_{+}$ is a complete graph. 
Then the following conditions are equivalent. 
\begin{enumerate}[(1)]
\item $\mathcal{B}(\Gamma)$ is free. 
\item $\Gamma_{-}$ is a threshold graph. 
\end{enumerate}
In this context, a simple graph is \textbf{threshold} if the graph does not contain graphs in Figure \ref{Fig:threshold} as induced subgraphs. 
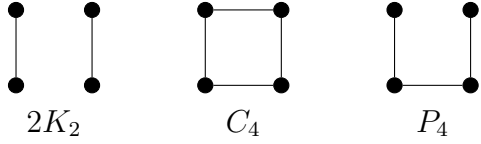
\begin{figure}[t]
\centering
\begin{tikzpicture}
\draw (0,1) node[v](1){};
\draw (0,0) node[v](2){};
\draw (1,0) node[v](3){};
\draw (1,1) node[v](4){};
\draw (0.5,-0.5) node[]{$ 2K_{2} $};
\draw (1)--(2);
\draw (3)--(4);
\end{tikzpicture}
\hspace{10mm}
\begin{tikzpicture}
\draw (0,1) node[v](1){};
\draw (0,0) node[v](2){};
\draw (1,0) node[v](3){};
\draw (1,1) node[v](4){};
\draw (0.5,-0.5) node[]{$ C_{4} $};
\draw (1)--(2)--(3)--(4)--(1);
\end{tikzpicture}
\hspace{10mm}
\begin{tikzpicture}
\draw (0,1) node[v](1){};
\draw (0,0) node[v](2){};
\draw (1,0) node[v](3){};
\draw (1,1) node[v](4){};
\draw (0.5,-0.5) node[]{$ P_{4} $};
\draw (1)--(2)--(3)--(4);
\end{tikzpicture}
\caption{Forbidden induced subgraphs for threshold graphs.}\label{Fig:threshold}
\end{figure} 
\end{corollary}

\begin{example}\label{Shi over F2}
Let $\Gamma$ be a signed graph and suppose that $\Gamma_{+}$ is complete. 
Then the affinographic arrangement $\mathcal{A}(\Gamma)$ over $\mathbb{F}_{2}$ is defined by 
\begin{align*}
\mathcal{A}(\Gamma) = \Set{\{x_{i}-x_{j}=0\} | 1 \leq i < j \leq \ell} \cup \Set{\{x_{i}-x_{j} = 1\} | \{i,j\} \in E_{\Gamma_{-}}}. 
\end{align*}
Then $\mathcal{A}(\Gamma)$ is an arrangement between the ``Coxeter arrangement" and the ``Shi arrangement" over $\mathbb{F}_{2}$. 
By Theorem \ref{main theorem signed} and Corollary \ref{Edelman-Reiner}, $\mathbf{c}\mathcal{A}(\Gamma)$ is free if and only if $\Gamma_{-}$ is threshold. 
\end{example}

\begin{remark}
Freeness of subarrangements between the Coxeter arrangement and the Shi arrangement of type $\mathrm{A}$ is characterized by Athanasiadis (Theorem \ref{Athanasiadis-Bailey free}), which is different from Example \ref{Shi over F2}. 
See \cite{palezzato2018free} and \cite{Tor22} for behavior of freeness for different characteristic of the base fields. 
\end{remark}

\section{Proof of Theorem \ref{main theorem 3dim}}\label{sec:3dim}

The proof of Theorem \ref{main theorem 3dim} requires the following lemmas. 

\begin{lemma}[{\cite[Corollary 3.3]{yoshinaga2005freeness-botlms}}]\label{lem:Yoshinaga}
Let $\mathcal{A}$ be an arrangement in dimension $3$ and $H \in \mathcal{A}$. 
Then the following conditions are equivalent 
\begin{enumerate}[(1)]
\item $\mathcal{A}$ is free. 
\item $\chi(\mathcal{A},t) = (t-1)(t-d_{1})(t-d_{2})$ and $\exp(\mathcal{A}^{H}, m^{H}) = (d_{1}, d_{2})$, 
\end{enumerate}
where $(\mathcal{A}^{H}, m^{H})$ denotes the \textbf{Ziegler restriction}, that is, 
\begin{align*}
m^{H}(X) \coloneqq \#\Set{K \in \mathcal{A}\setminus\{H\} | K \cap H = X}. 
\end{align*}
\end{lemma}

\begin{lemma}[{\cite[Proposition 1.23]{yoshinaga2014freeness-adlfdsdtm}}, {\cite[Theorem 1.5]{wakamiko2007exponents-tjom}}]\label{lem:2-dim exponents}
Let $\mathcal{A} = \{H_{1}, \dots, H_{n}\}$ be an arrangement in dimension $2$. 
Let $m$ be a multiplicity on $\mathcal{A}$ with $m_{i} \coloneqq m(H_{i})$ and $m_{1} \geq m_{2} \geq \dots \geq m_{n}$. 
Then the following conditions hold. 
\begin{enumerate}[(1)]
\item\label{lem:2-dim exponents n} If $n \geq \frac{|m|}{2}+1$, then $\exp(\mathcal{A},m) = (|m|-n+1, n-1)$. 
\item\label{lem:2-dim exponents wakamiko} If $n=3$ and the characteristic of the base field is $0$, then 
\begin{align*}
\exp(\mathcal{A},m) = \begin{cases}
(m_{2}+m_{3}, m_{1}) & \text{ if } m_{1} \geq m_{2}+m_{3}, \\
(k,k) & \text{ if } m_{1} \leq m_{2}+m_{3} \text{ and } |m| = 2k, \\
(k,k+1) & \text{ if } m_{1} \leq m_{2}+m_{3} \text{ and } |m| = 2k+1. 
\end{cases}
\end{align*}
\end{enumerate}
\end{lemma}

\begin{lemma}\label{lem:exponents}
Suppose that $q \in \mathbb{C}^{\times}$ is a transcendental number or a positive real number other than $1$ and let $(\mathcal{A}, m)$ be a multiarrangement in $\mathbb{C}^{2}$ determined by 
\begin{align*}
x_{1}^{s+1}x_{2}^{t+1}\prod_{g \in \Lambda}(x_{1}-q^{g}x_{2}), 
\end{align*}
where $\Lambda$ is a finite subset of $\mathbb{Z}$ and $0 \leq s \leq t \leq u \coloneqq |\Lambda|$. 
Then 
\begin{align*}
\exp(\mathcal{A},m) = \begin{cases}
(s+t+1, u+1) & \text{ if } u \geq s+t, \\
(k+1,k+1) & \text{ if } u \leq s+t \text{ and } s+t+u = 2k, \\
(k+1,k+2) & \text{ if } u \leq s+t \text{ and } s+t+u = 2k+1. 
\end{cases}
\end{align*}
\end{lemma}
\begin{proof}
First, suppose that $u \geq s+t$. 
In order to use Lemma \ref{lem:2-dim exponents}\ref{lem:2-dim exponents n}, let $n \coloneqq |\mathcal{A}| = u+2$. 
Then
\begin{align*}
n-\dfrac{|m|}{2}-1 = u+2-\dfrac{s+t+u+2}{2}-1
= \dfrac{u-s-t}{2} \geq 0. 
\end{align*}
Therefore $\exp(\mathcal{A},m) = (|m|-n+1,n-1) = (s+t+1,u+1)$. 

Next, suppose that $u \leq s+t$ and $s+t+u = 2k$. 
Let $\exp(\mathcal{A},m) = (d_{1},d_{2})$ and $d_{1} \leq d_{2}$. 
It satisfies to show that $d_{1} \geq k+1$ since $d_{1} + d_{2} = |m| = s+t+u+2 = 2(k+1) \leq 2d_{1}$ implies $d_{1} = d_{2} = k+1$. 
Assume that $\theta \in D(\mathcal{A},m)$ is a homogeneous element and $d \coloneqq \deg \theta \leq k$. 
We will show that $\theta = 0$. 

Since $\theta(x_{1}) \in x_{1}^{s+1}S$ and $\theta(x_{2}) \in x_{2}^{t+1}S$, there exist polynomials $f,g \in S$ such that 
\begin{align*}
\theta = x_{1}^{s+1}f\partial_{1} + x_{2}^{t+1}g\partial_{2}, 
\end{align*}
where $\partial_{1}$ denotes the derivation $\frac{d}{dx_{i}}$. 
Write $f$ and $g$ as 
\begin{align*}
f = \sum_{i=0}^{d-s-1}b_{i}x_{1}^{d-s-1-i}x_{2}^{i} 
\quad \text{ and } \quad 
g = \sum_{i=0}^{d-t-1}c_{i}x_{1}^{d-t-1-i}x_{2}^{i}. 
\end{align*}

Let $\Lambda = \{g_{1}, \dots, g_{u}\}$. 
To ease notation write $q_{i} \coloneqq q^{g_{i}}$. 
Then in modulus $x_{1}-q_{i}x_{2}$
\begin{align*}
0 &\equiv \theta(x_{1}-q_{i}x_{2}) \\
&= x_{1}^{s+1}f - q_{i}x_{2}^{t+1}g \\
&\equiv x_{2}^{d}\left( b_{0}q_{i}^{d}+b_{1}q_{i}^{d-1} + \dots + b_{d-s-1}q_{i}^{s+1} -c_{0}q_{i}^{d-t} - c_{1}q_{i}^{d-t-1} - \dots - c_{d-t-1}q_{i} \right). 
\end{align*}
Note that the powers in the last expression distinct since 
\begin{align*}
(s+1)-(d-t) 
&= s+1-d+t \\
&\geq s+1-k+t \\
&= \dfrac{1}{2}(2s+2-s-t-u+2t) \\
&= \dfrac{1}{2}(s+t-u+2) > 0. 
\end{align*}
Therefore 
\begin{align*}
\lambda \coloneqq (d, d-1, \dots, s+1, d-t, d-t-1, \dots, 1)
\end{align*}
is decreasing and defines an integer partition. 
The length $\ell(\lambda)$ satisfies 
\begin{align*}
\ell(\lambda) = (d-s) + (d-t) = 2d-s-t
\leq 2k-s-t = u. 
\end{align*}
Thus we have the following linear equation. 
\begin{align*}
\begin{pmatrix}
q_{1}^{d} & \dots & q_{1}^{s+1} & q_{1}^{d-t} & \dots & q_{1} \\
q_{2}^{d} & \dots & q_{2}^{s+1} & q_{2}^{d-t} & \dots & q_{2} \\
\vdots & & \vdots & \vdots & & \vdots \\
q_{\ell(\lambda)}^{d} & \dots & q_{\ell(\lambda)}^{s+1} & q_{\ell(\lambda)}^{d-t} & \dots & q_{\ell(\lambda)}
\end{pmatrix}\begin{pmatrix}
b_{0} \\ \vdots \\ b_{d-s-1} \\ -c_{0} \\ \vdots \\ -c_{d-t-1}
\end{pmatrix} = 0. 
\end{align*}
Let $a_{\lambda}(q_{1}, \dots, q_{\ell(\lambda)})$ denote the determinant of the coefficient matrix. 
By the definition of the Schur polynomial $s_{\lambda}$, we have 
\begin{align*}
a_{\lambda}(q_{1}, \dots, q_{\ell(\lambda)}) = s_{\lambda}(q_{1}, \dots, q_{\ell(\lambda)}) \ \Delta(q_{1}, \dots, q_{\ell(\lambda)}), 
\end{align*}
where $\Delta$ denotes the Vandermonde determinant. 
Since $q \in \mathbb{C}^{\times}$ is not a root of unity, $q_{1}, \dots, q_{\ell(\lambda)}$ are distinct and hence $\Delta(q_{1}, \dots, q_{\ell(\lambda)}) \neq 0$. 
It is well known that the Schur polynomial is a positive linear combination of monomials. 
Therefore $s_{\lambda}(q_{1}, \dots, q_{\ell(\lambda)}) \neq 0$ if $q$ is a positive real number or a transcendental number. 
In this case we have $a_{\lambda}(q_{1}, \dots, q_{\ell(\lambda)}) \neq 0$, which implies $\theta = 0$. 

The remaining case is similar. 
We leave the proof for the reader. 
\end{proof}

Now, we are ready to prove Theorem \ref{main theorem 3dim}. 

\begin{theorem}[Restatement of Theorem \ref{main theorem 3dim}]
Let $\Gamma$ be a simple gain graph on $3$ vertices with gain group $G = \mathbb{Z}$. 
Suppose that $q$ is a transcendental number or a positive real number other than $1$. 
Then $\mathrm{c}\mathcal{A}(\Gamma)$ is free if and only if $\mathcal{B}(\Gamma)$ is free. 
\end{theorem}
\begin{proof}
Define subsets $\Lambda_{1}, \Lambda_{2}$, and $\Lambda_{3}$ by 
\begin{align*}
\Lambda_{1} &\coloneqq \Set{g \in \mathbb{Z} | [1,2,g] \in E_{\Gamma}}, \\
\Lambda_{2} &\coloneqq \Set{g \in \mathbb{Z} | [2,3,g] \in E_{\Gamma}}, \\
\Lambda_{3} &\coloneqq \Set{g \in \mathbb{Z} | [1,3,g] \in E_{\Gamma}}. 
\end{align*}
Let $m_{i} = |\Lambda_{i}|$ for each $i \in \{1,2,3\}$. 
We may assume that $m_{1} \geq m_{2} \geq m_{3}$. 
The defining polynomials of $\mathrm{c}\mathcal{A}(\Gamma)$ and $\mathcal{B}(\Gamma)$ are 
\begin{align*}
Q(\mathrm{c}\mathcal{A}(\Gamma)) &= z\prod_{g \in \Lambda_{1}}(x_{1}-x_{2}-gz)\prod_{g \in \Lambda_{2}}(x_{2}-x_{3}-gz)\prod_{g \in \Lambda_{3}}(x_{1}-x_{3}-gz),  \\
Q(\mathcal{B}(\Gamma)) &= x_{1}x_{2}x_{3}\prod_{g \in \Lambda_{1}}(x_{1}-q^{g}x_{2})\prod_{g \in \Lambda_{2}}(x_{2}-q^{g}x_{3})\prod_{g \in \Lambda_{3}}(x_{1}-q^{g}x_{3}). 
\end{align*}
Taking the Ziegler restrictions of $\mathrm{c}\mathcal{A}(\Gamma)$ and $\mathcal{B}(\Gamma)$ with ${z=0}$ and $\{x_{3}=0\}$ respectively yields 
\begin{align*}
Q(\mathrm{c}\mathcal{A}(\Gamma)^{\{z=0\}}, m_{\mathcal{A}}) &= (x_{1}-x_{2})^{m_{1}}(x_{2}-x_{3})^{m_{2}}(x_{1}-x_{3})^{m_{3}}, \\
Q(\mathcal{B}(\Gamma)^{\{x_{3}=0\}}, m_{\mathcal{B}}) &= 
x_{1}^{m_{3}+1}x_{2}^{m_{2}+1}\prod_{g \in \Lambda_{1}}(x_{1}-q^{g}x_{2}). 
\end{align*}

By Lemma \ref{lem:equalbetweencharpoly}, \ref{lem:Yoshinaga}, \ref{lem:2-dim exponents}, and \ref{lem:exponents}, 
\begin{align*}
& \hspace{10.5mm} \mathrm{c}\mathcal{A}(\Gamma) \text{ is free.} \\
&\Longleftrightarrow \chi(\mathrm{c}\mathcal{A}(\Gamma), t) = t(t-1)(t-d_{1})(t-d_{2}) \\
& \hspace{60mm} \text{ and } \exp(\mathrm{c}\mathcal{A}(\Gamma)^{\{z=0\}}, m_{\mathcal{A}}) = (d_{1},d_{2}). \\
&\Longleftrightarrow \chi(\mathcal{B}(\Gamma),t) = (t-1)(t-d_{1}-1)(t-d_{2}-1) \\
& \hspace{60mm} \text{ and } \exp(\mathcal{B}(\Gamma)^{\{x_{3}=0\}}, m_{\mathcal{B}}) = (d_{1}+1, d_{2}+1). \\
&\Longleftrightarrow \mathcal{B}(\Gamma) \text{ is free.} 
\end{align*}
\end{proof}

\section*{Acknowledgment}
This work was supported by JSPS KAKENHI Grant Number JP	21K13777.

\bibliographystyle{amsalpha}
\bibliography{bibfile}

\end{document}